\documentclass[10pt,electronic]{amsart}                                                 
\usepackage{graphicx}
\usepackage[hmargin=3.6cm,vmargin=3.2cm]{geometry} 
\usepackage[applemac]{inputenc}
\usepackage[T1]{fontenc}
\usepackage[english]{babel}
\usepackage[small,nohug]{diagrams}
\usepackage{url}
\usepackage{amsmath}
\usepackage{amsthm}                      
\usepackage{amssymb}
\usepackage{mathrsfs}
\usepackage{enumerate}

\usepackage[isbn=false,doi=false,eprint=false,url=false,style=authoryear,dashed=false]{biblatex}
\makeatletter
\newrobustcmd*{\parentexttrack}[1]{%
  \begingroup
  \blx@blxinit
  \blx@setsfcodes
  \blx@bibopenparen#1\blx@bibcloseparen
  \endgroup}
\AtEveryCite{%
  }
\makeatother
\renewcommand{\cite}{\parencite}
\DeclareNameAlias{sortname}{last-first}
\renewbibmacro*{volume+number+eid}{%
  \printfield{volume}%
  \setunit*{\addnbspace}% NEW (optional); there's also \addnbthinspace
  \printfield{number}%
  \setunit{\addcomma\space}%
  \printfield{eid}}
 \DeclareFieldFormat[article]{volume}{\textbf{#1}}
\DeclareFieldFormat[article]{number}{{\mkbibparens{#1}},}
\DeclareFieldFormat{pages}{#1 pages}
\DeclareFieldFormat
  [article,inbook,incollection,inproceedings,patent,thesis,unpublished]
  {pages}{#1}
\renewbibmacro{in:}{}

\theoremstyle{plain}                                                           
\newtheorem{thm}{Theorem}[section]

\newtheorem{prop}[thm]{Proposition}

\theoremstyle{definition}
\newtheorem{ex}[thm]{Example}
\newtheorem{rem}[thm]{Remark}

\DeclareMathOperator{\SL}{SL}
\DeclareMathOperator{\GL}{GL}
\DeclareMathOperator{\Spec}{Spec}
\DeclareMathOperator{\Sym}{Sym}                  
\DeclareMathOperator{\Sp}{Sp}
\DeclareMathOperator{\Tr}{Tr}
\DeclareMathOperator{\sgn}{sgn}

\newcommand{\gr}{\mathfrak{gr}}
\newcommand{\field}[1]{\ensuremath{\mathbf{#1}}}
\newcommand{\R}{\ensuremath{\field{R}}}        
\newcommand{\Q}{\ensuremath{\field{Q}}}        
\newcommand{\C}{\ensuremath{\field{C}}}

\newcommand{\Z}{\ensuremath{\field{Z}}}

\newcommand{\A}{\mathcal{A}}
\renewcommand{\AA}{\field{A}}
\newcommand{\RR}{\mathrm{Rep}}
\newcommand{\gen}{\mathrm{gen}}
\renewcommand{\H}{\mathscr H}
\newcommand{\V}{\mathbb{V}}
\renewcommand{\S}{\mathsf{S}}
\newcommand{\GSp}{\mathrm{GSp}}
\newcommand{\Eis}{\mathrm{Eis}}
\newcommand{\PGSp}{\mathrm{PGSp}}
\newcommand{\PGL}{\mathrm{PGL}}
\newcommand{\Gal}{\mathrm{Gal}}
\newcommand{\g}{\mathfrak{g}}
\newcommand{\cusp}{\mathrm{cusp}}
\newcommand{\IH}{H}
\newcommand{\f}{\mathrm{fin}}
\newcommand{\X}{\mathscr X}
\newcommand{\VV}{\mathscr V} 
\newcommand{\bA}{\widetilde{\A}}

\title[Local systems on the moduli of principally polarized abelian surfaces]{Cohomology of local systems on the moduli of principally polarized abelian surfaces}
\author{Dan Petersen}
\thanks{This work was carried out at KTH Royal Institute of Technology, supported by the G\"oran Gustafsson foundation, and in the group of Pandharipande at ETH Z\"urich, supported by grant ERC-2012-AdG-320368-MCSK}
\email{danpete@math.kth.se}
\address{Departement Mathematik \\ ETH Z\"urich \\  R\"amistrasse 101 \\ 8092 Z\"urich\\ Switzerland} 
\address{ Institutionen f\"or Matematik \\ KTH Royal Institute of Technology \\ 100 44 Stockholm \\ Sweden }

\begin{document} 
 \maketitle   
 
 \begin{abstract}
 Let $\A_2$ be the moduli stack of principally polarized abelian surfaces. Let $\V$ be a smooth $\ell$-adic sheaf on $\A_2$ associated to an irreducible rational finite dimensional representation of $\Sp(4)$. We give an explicit expression for the cohomology of $\V$ in any degree in terms of Tate type classes and Galois representations attached to elliptic and Siegel cusp forms. This confirms a conjecture of Faber and van der Geer. As an application we prove a dimension formula for vector-valued Siegel cusp forms for $\Sp(4,\Z)$ of weight three, which had been conjectured by Ibukiyama.

 \end{abstract}
 
\section{Introduction}

Let $Y=\Gamma \backslash \mathfrak{H}$ be a modular curve, given by the quotient of the upper half plane by a congruence subgroup $\Gamma \subset \SL(2,\Z)$. An irreducible rational representation $\V$ of $\SL(2)$ defines 
a local system on $Y$, since $\V$ is in particular a representation of $\pi_1(Y) \cong \Gamma \subset \SL(2)$. After work of Eichler, Shimura, Ihara, Deligne, and many others after them, we understand extremely well the cohomology groups $H^\bullet(Y,\V)$. The cohomology classes can be described group-theoretically in terms of modular forms for the group $\Gamma$, and it has a (split) mixed Hodge structure in which the pure part corresponds to cusp forms and its complement to Eisenstein series. We can think of $\V$ also as a smooth $\ell$-adic sheaf (and $Y$ as defined over a number field, or a deeper arithmetic base), in which case the \'etale cohomology $H^\bullet(Y,\V)$ can be expressed in terms of Galois representations attached to the same modular forms \cite{deligne69}. 

There is a vast theory describing the generalization of the above to moduli spaces of higher-dimensional abelian varieties with some extra structure (polarization, endomorphism, and level), and to more general Shimura varieties. But there is not a single example where our understanding is as complete as in genus one. 

In this article we consider one of the simplest higher-genus examples and give a quite explicit description of the cohomology in this case. Namely,  consider the moduli space $\A_2$ of principally polarized abelian surfaces, and let $\V$ be a smooth $\ell$-adic sheaf associated to an irreducible representation of $\Sp(4)$. The main theorem of this article is an explicit expression for the (semi-simplification of the) $\ell$-adic Galois representation $H^k_c(\A_2,\V)$
for any $k$ and any $\V$ in terms of Tate type classes and Galois representations attached to level $1$ elliptic/Siegel cusp forms. 

These cohomology groups are natural objects of study for algebraic geometers, in particular because of applications to moduli of curves. In particular, the results of this paper are used in \cite{m28ct} to prove that the Gorenstein conjecture fails for the tautological rings of the spaces $\mathcal M_{2,n}^{\mathsf{ct}}$ for $n \geq 8$. There is some history of algebraic geometers studying the cohomology  of $\V_{a,b}$ for small values of $a+b$ by ad hoc methods for such applications, see e.g.\ \cite[Section 8]{getzlergenustwo}, \cite{bergstrom09}, \cite[Section 3]{petersentommasi}. Let us also mention \cite{fvdg1} who used point counts over finite fields to conjecture an expression for the virtual $\ell$-adic Galois representation $$ \sum_k (-1)^k [H^k_c(\A_2,\V)] \in K_0(\mathsf{Gal})$$
for any $\V_{a,b}$; see also \cite[Section 6]{bfg11} for a more detailed description. The results in this paper confirm Faber and van der Geer's conjecture. When $\V$ has regular highest weight, their conjecture was proven in \cite{weissauer} (and later independently in \cite{tehraniendoscopy}).

Using the BGG-complex of Faltings, one can relate the results of this paper to the coherent cohomology of the bundles of Siegel modular forms for $\Sp(4,\Z)$, as we explain at the end of Section \ref{mainthm}. A direct consequence of our main theorem is a proof of a dimension formula for vector-valued Siegel modular forms for $\Sp(4,\Z)$ of weight $3$, which had been conjectured in \cite{ibukiyama1}. This result has been independently obtained in \cite{taibi} using Arthur's trace formula.

The strategy of our proof is as follows. Up to semi-simplification, the cohomology is the direct sum of the \emph{Eisenstein cohomology} and the \emph{inner cohomology}. The Eisenstein cohomology on $\A_2$ of an arbitrary local system was determined in \cite{harder}, so we need only to find the inner cohomology. Now we use that the inner cohomology contains the cuspidal cohomology and is contained in the intersection cohomology, and both of these can be understood in terms of data attached to discrete spectrum automorphic representations for $\GSp(4)$. There is a very large body of work dealing with automorphic representations on $\GSp(4)$ (due to Piatetski-Shapiro, Soudry, Arthur, Weissauer, Taylor, Hales, Waldspurger and many others) since it is one of the first test cases for the general Langlands program. Since we will only work in level $1$, we can work with $\PGSp(4)$, in which case all necessary information on the discrete spectrum automorphic representations is worked out and described very explicitly in  \cite{flicker}. These results allow us to determine both the cuspidal and the intersection cohomology of these local systems, and to deduce after comparing with Harder's results that the inner cohomology coincides with the cuspidal cohomology in these cases. 

In Section 2 of this article I state the main theorem and explain the applications to vector-valued Siegel cusp forms. Section 3 contains a brief review of automorphic representations and the cohomology of Shimura varieties. I hope that this will help make the arguments accessible for algebraic geometers without this background. Section 4 specializes to $\PGSp(4)$ and contains the proof of the main theorem.

I am grateful to Jonas Bergstr\"om for many useful discussions on these topics and for his interest in this work, and Tomoyoshi Ibukiyama for several helpful pointers to the literature. 

\section{Statement of results}\label{mainresults}

Let $\A_2$ denote the moduli stack of principally polarized abelian surfaces. Let $f \colon \X \to \A_2$ be the universal family. We have a local system (smooth $\ell$-adic sheaf) $\V = \mathrm R^1 f_\ast \Q_\ell$ on $\A_2$ of rank $4$ and weight $1$, and there is a symplectic pairing 
$$ \wedge^2 \V \to \Q_\ell(-1).$$
Here $\Q_\ell(-1)$ denotes the Tate twist of the constant local system on $\A_2$. Recall Weyl's construction of the irreducible representations of $\Sp(4)$ \cite[Section 17.3]{fh91}: if $V$ is the standard $4$-dimensional symplectic vector space, then the irreducible representation with highest weight $a \geq b \geq 0$ is a constituent of $V^{\otimes (a+b)}$, where it is `cut out' by Schur functors and by contracting with the symplectic form. For instance, the representation of highest weight $(2,0)$ is $\Sym^2(V)$, and the representation $(1,1)$ is the complement of the class of the symplectic form inside $\wedge^2 V$. Weyl's construction works equally well in families, and so for each $a \geq b \geq 0$ we obtain a local system $\V_{a,b}$ which is a summand in $\V^{\otimes (a+b)}$. In this paper we determine the cohomology of $\V_{a,b}$ considered as an $\ell$-adic Galois representation up to semi-simplification. 

Note that every point of $\A_2$ has the automorphism $(-1)$, given by inversion on the abelian variety. This automorphism acts as multiplication by $(-1)^{a+b}$ on the fibers of $\V_{a,b}$. This shows that the local system has no cohomology when $a+b$ is odd.  Hence we restrict our attention to the case when $a+b$ is even. 

Before we can state our main results we need to introduce some notation. For any $k$, let $s_k$ denote the dimension of the space of cusp forms for $\SL(2,\Z)$ of weight $k$. Similarly for any $j\geq 0$, $k\geq 3$ we denote by $s_{j,k}$ the dimension of the space of vector-valued Siegel cusp forms for $\Sp(4,\Z)$, transforming according to the representation $\Sym^{ j} \otimes \det^{ k}$. 

To each normalized cusp eigenform $f$ for $\SL(2,\Z)$ of weight $k$ is attached a $2$-dimensional $\ell$-adic Galois representation $\rho_f$ of weight $k-1$ \cite{deligne69}. We define 
$\S_k = \bigoplus_{f} \rho_f$
to be the direct sum of these Galois representation for fixed $k$. By the main theorem of \cite{weissauer4d} there are also $4$-dimensional Galois representations attached to vector-valued Siegel cusp eigenforms for $\Sp(4,\Z)$ of type $\Sym^j \otimes \det^k$ with $k \geq 3$, and we define $\S_{j,k}$ analogously. So $\dim \S_k = 2s_k$ and $\dim \S_{j,k} = 4s_{j,k}$. 

Moreover, we introduce $s_k'$: this is the cardinality of the set of normalized cusp eigenforms $f$ of weight $k$ for $\SL(2,\Z)$, for which the central value $L(f,\frac 1 2)$ vanishes. In this paper all $L$-functions will be normalized to have a functional equation relating $s$ and $1-s$. The functional equation shows that the order of $L(f,s)$ at $s=\frac 1 2$ is always odd if $k \equiv 2 \pmod 4$ and even if $k \equiv 0 \pmod 4$. Hence in the former case $s_k = s_k'$; in the latter case $0 \leq s_k' \leq s_k$. In our results, the quantity $s_k'$ will only occur in the case $k \equiv 0 \pmod 4$, and in this case it is conjectured that $s_k' = 0$. Indeed, \cite{conreyfarmer} proved that this vanishing is implied by Maeda's conjecture; Maeda's conjecture has been verified numerically for weights up to $14000$ \cite{ghitza-mcandrew}.

Finally we define $\overline\S_{j,k} = \gr^W_{j+2k-3}\S_{j,k}$; in other words, we consider only the part of $\S_{j,k}$ which satisfies the Ramanujan conjecture. Counterexamples to the Ramanujan conjecture arise from the Saito--Kurokawa lifting: for a cusp eigenform $f$ of weight $2k$ for $\SL(2,\Z)$, where $k$ is odd, there is attached a scalar valued Siegel cusp form of weight $k+1$ for $\Sp(4,\Z)$ whose attached $\ell$-adic Galois representation has the form 
$$ \Q_\ell (-k+1) \oplus \rho_f \oplus \Q_\ell(-k) $$
where $\rho_f$ is the Galois representation of weight $2k-1$ attached to $f$. By \cite[Theorem 3.3]{weissauer}, these are in fact the only Siegel cusp forms violating the Ramanujan conjecture. Thus $\overline{\S}_{j,k} = \S_{j,k}$ unless $j=0$ and $k$ is even, in which case $\overline{\S}_{j,k}$ is obtained from $\S_{j,k}$ by removing the two summands of Tate type from each Saito--Kurokawa lift. 

Note that the definitions of $s_k$, $\S_k$, $s_{j,k}$ and $\S_{j,k}$ used in \cite{fvdg1} are different from ours: theirs is not only a sum over cusp forms, but includes in the case $k=2$ (resp. $j=0$, $k=3$) the contribution from the trivial automorphic representation. This allows for a compact expression for the virtual Galois representation $\sum_i (-1)^i [H^i_c(\A_2,\V_{a,b})]$ but will not be used here.

\begin{thm}\label{mainthm}Suppose $(a,b) \neq (0,0)$, and that $a+b$ is even. Then: \begin{enumerate}[\scshape(1)]
\item $H^k_c(\A_2,\V_{a,b})$ vanishes for $k \notin \{2,3,4\}$.
\item In degree $4$ we have $$ H^4_c (\A_2,\V_{a,b}) = \begin{cases} s_{a+b+4} \Q_\ell(-b-2) & a=b \text{ even,} 
\\ 0  & \text{otherwise.}   \end{cases} $$
\item In degree $3$ we have, up to semi-simplification, 
\begin{align*} H^3_c(\A_2,\V_{a,b}) &= \overline\S_{a-b,b+3} \\
& + s_{a+b+4}\S_{a-b+2}(-{b-1}) \\
&+  \S_{a+3}  \\
%&+  \begin{cases} \S_{a+3} & a \text{ odd,} \\ 0 & \text{otherwise,}\end{cases} \\
&+  \begin{cases} 
s_{a+b+4}' \Q_\ell(-{b-1}) & a=b  \text{ even,}  \\
s_{a+b+4} \Q_\ell(-{b-1}) & \text{otherwise,} \end{cases}\\
&+ \begin{cases} \Q_\ell & a=b \text{ odd,} \\ 0 & \text{otherwise,}\end{cases} \\ 
&+ \begin{cases} \Q_\ell(-1) & b=0, \\ 0 & \text{otherwise.}  \end{cases} .
\end{align*}
\item In degree $2$ we have, again up to semi-simplification, that
\begin{align*} H^2_c(\A_2,\V_{a,b}) &= 
 \S_{b+2} \\
%\begin{cases} \S_{b+2} & a \text{ even,} \\ 0 &  \text{otherwise,}\end{cases}\\
&+  s_{a-b+2} \Q_\ell\\
&+  \begin{cases} 
s_{a+b+4}'  \Q_\ell(-{b-1}) & a=b \text{ even,}  \\
0 & \text{otherwise,} \end{cases} \\
&+ \begin{cases} \Q_\ell & a > b > 0 \text{ and } a, b \text{ even,} \\ 0 & \text{otherwise.}  \end{cases} 
\end{align*}

\end{enumerate}
\end{thm}

To exemplify the notation: $s_{a+b+4}\S_{a-b+2}(-{b-1})$ means a direct sum of $s_{a+b+4}$ copies of the Galois representation $\S_{a-b+2}$, Tate twisted $b+1$ times. 

As remarked earlier, it is conjectured that both occurrences of $s_k'$ in the above theorem can be replaced by $0$.

\begin{rem}\label{hodgetheory} It will be clear from the proof that the result is valid (and even a bit easier) also in the category of mixed Hodge structures. Harder's computation of the Eisenstein cohomology is valid in this category, and our computation of the inner cohomology identifies it with the cuspidal cohomology, which obtains a natural Hodge structure from the bigrading on $(\g,K)$-cohomology. This bigrading is compatible with the one obtained using the `filtration b\^ete' and the BGG-complex of \cite[Theorem VI.5.5.]{faltingschai}.\end{rem}

\begin{rem} \label{motivicremark} It is conjectured that the Galois representations $H^k_c(\A_2,\V_{a,b})$ are not semisimple in general. Suppose that $a=b = 2k-1$. Then our expression for the semisimplification of $H^3_c(\A_2,\V_{a,b})$ contains the terms $s_{4k+2} \Q_\ell(-2k)$ and $\S_{4k+2}$, the latter being the `Saito--Kurokawa' summand of $\overline \S_{0,2k+2}$. According to a conjecture in \cite[81--82]{harderbook}, these should form a nontrivial extension:
$$ 0 \to s_{4k+2}\Q_\ell(-2k) \to M \to \S_{4k+2} \to 0.$$
Note that if $f$ is a Hecke eigenform of weight $4k+2$ and $\rho_f$ is the attached Galois representation (or `motive'), then conjectures of Deligne--Bloch--Beilinson \cite[Section 1]{grosscentralvalue} predict that 
$$\dim \mathrm{Ext}^1(\Q_\ell(-2k),\rho_f) = \mathrm{ord}_{s=\tfrac 1 2} L(f,s), $$
and the functional equation for $L(f,s)$ forces it to vanish at $s=\tfrac 1 2$. Here the Ext-group is computed either in the category of $\ell$-adic Galois representations, or (even better) in the category of mixed motives. I do not know whether there exists a cusp form for the full modular group whose $L$-function vanishes to more than first order at the central point.    \end{rem}

\subsection{Application to dimension formulas for Siegel modular forms} A consequence of Remark \ref{hodgetheory} is that our main theorem can be applied to produce dimension formulas for vector-valued Siegel modular forms.  Let $i \colon \A_2 \hookrightarrow \bA_2$ be a toroidal compactification. Let $\VV_{j,k}$  for $j,k \in \Z$, $j \geq 0$  be the vector bundle on $\bA_2$ whose global sections are vector-valued Siegel modular forms of type $\Sym^{ j} \otimes \det^{ k}$. Let similarly $\VV_{j,k}(-D_\infty)$ be the vector bundle of Siegel cusp forms. The \emph{BGG-complex} (resp.\ the \emph{dual BGG-complex}) is a resolution of $i_\ast \V_{a,b} \otimes \C$ (resp.\ $i_! \V_{a,b} \otimes \C$) in terms of the vector bundles $\VV_{j,k}$ (resp.\ $\VV_{j,k}(-D_\infty)$). Then \cite[Theorem VI.5.5]{faltingschai} asserts that the hypercohomology spectral sequence of the BGG-complex degenerates, and that the Hodge filtration on the cohomology of $\V_{a,b}$ can be defined in terms of a filtration of the BGG-complex. There is also an analogous statement for the dual BGG-complex and the compactly supported cohomology. Specialized to our case, their theorem (in the case of the dual BGG-complex) asserts the following (see \cite[Theorem 17]{getzlergenustwo}): 
\begin{thm}[Faltings--Chai] The cohomology groups $H^\bullet_c(\A_2,\V_{a,b} \otimes \C)$ have a Hodge filtration with Hodge numbers in the set $\{a+b+3,a+2,b+1,0\}$. The associated graded pieces satisfy
\begin{align*}
\gr_F^0 H^\bullet_c(\A_2,\V_{a,b} \otimes \C) & \cong H^\bullet(\bA_2,\VV_{a-b,-a}(-D_\infty)), \\
\gr_F^{b+1} H^\bullet_c(\A_2,\V_{a,b} \otimes \C) & \cong H^{\bullet-1}(\bA_2,\VV_{a+b+2,-a}(-D_\infty)), \\
\gr_F^{a+2} H^\bullet_c(\A_2,\V_{a,b} \otimes \C) & \cong H^{\bullet-2}(\bA_2,\VV_{a+b+2,1-b}(-D_\infty)), \\
\gr_F^{a+b+3} H^\bullet_c(\A_2,\V_{a,b} \otimes \C) & \cong H^{\bullet-3}(\bA_2,\VV_{a-b,b+3}(-D_\infty)). 
\end{align*} \end{thm}
We record three immediate consequences of this theorem combined with our main theorem. The first of these is a proof of a conjecture of Ibukiyama, whereas the second two are new proofs of results which are already known (by admittedly much more direct arguments).
\begin{enumerate}
\item {The bundles $\VV_{j,k}(-D_\infty)$ have no higher cohomology for any $j \geq 0$, $k \geq 3$, with the sole exception of $H^3(\bA_2,\VV_{0,3}(-D_\infty)) \cong \C$.} (To prove this, consider $\gr_F^{a+b+3}$.) An explicit formula for the Euler characteristic of the vector bundles $\VV_{j,k}(-D_\infty)$ was calculated in \cite{tsushima} using Hirzebruch--Riemann--Roch; thus, we obtain a dimension formula for vector-valued Siegel cusp forms for all $j \geq 0$, $k \geq 3$. Tsushima himself proved that these bundles have no higher cohomology when $k \geq 5$ using the Kawamata--Viehweg vanishing theorem, and conjectured that it can be improved to $k \geq 4$. The fact that this vanishing result can be extended to $k \geq 3$ is particular to the case of the full modular group and was conjectured in \cite[Conjecture 2.1]{ibukiyama1}. The resulting dimension formula for $k=3$ can be stated as
$$ \sum_{j \geq 0} s_{j,3}x^j = \frac{x^{36}}{(1-x^6)(1-x^8)(1-x^{10})(1-x^{12})}. $$
This result has also been proven in \cite[Section 5]{taibi}. 
\item There are no vector-valued Siegel modular forms of weight $1$ for the full modular group. (Put $b=0$ and consider $\gr_F^{a+2}$ to prove the case $\Sym^j \otimes \det$ with $j \geq 2$; the cases $j < 2$ require a separate  [easy] argument.) This result was previously known by \cite[Theorem 6.1]{ibukiyamaproceedings}.
\item The Siegel $\Phi$-operator is surjective for any $j \geq 0, k \geq 3$. Recall that the $\Phi$-operator maps
Siegel modular forms of type $\Sym^j \otimes \det^k$ to elliptic modular forms of weight $j+k$, and that the image of $\Phi$ consists only of cusp forms if $j > 0$. Now, the dimension of the part of $ \gr_F^{a+b+3} H^3(\A_2,\V_{a,b} \otimes \C)$ given by Eisenstein cohomology is exactly the dimension of the image of the $\Phi$-operator for $\Sym^{a-b} \otimes \det^{b+3}$, since the part given by inner cohomology coincides with the dimension of the space of cusp forms. But the dimension of this part of Eisenstein cohomology is $s_{a+3}$ unless $a=b$ is odd, in which case it is $s_{a+3}+1$. The result follows from this. Surjectivity of the $\Phi$-operator is known more generally for arbitrary level when $k \geq 5$ and $j > 0$ by \cite{arakawa}. The scalar valued case is a classical theorem of Satake. The case $k=4$ (and $k=2$) is \cite[Theorem 5.1]{ibukiyamawitt}.

\end{enumerate}
Only Siegel modular forms of weight two are inaccessible via the cohomology of local systems. In a sequel to this paper we will use similar arguments to derive dimensional results for Siegel modular forms with nontrivial level.

\section{R\'esum\'e of automorphic representations}

In this section I briefly recall some (mostly standard) facts from the theory of automorphic representations that are needed for this paper. Rather than providing detailed references everywhere, I will give general references at the beginning of each subsection.

\subsection{Automorphic representations}\label{auto} \cite{boreljacquet,cogdellkimmurty} Let $G$ be a reductive connected group over $\Q$. Let $\AA = \AA_\f \times \R$ be the ring of (rational) ad\`eles.  Let $Z$ be the center of $G$, and $\omega$ a unitary character of $Z(\AA)/Z(\Q)$. We define $L^2(G(\Q)\backslash G(\AA),\omega)$ to be the space of measurable functions $f$ on $G(\Q)\backslash G(\AA)$ which are square integrable with respect to a translation invariant measure, and which satisfy 
$f(zg) = \omega(z)f(g)$ for any $z \in Z(\AA)$. The group $G(\AA)$ acts on this space by right translation. A representation of $G(\AA)$ is called \emph{automorphic} if it is a  subquotient of  $L^2(G(\Q)\backslash G(\AA),\omega)$, for some $\omega$. We call $\omega$ the \emph{central character} of the automorphic representation.

The space $L^2(G(\Q)\backslash G(\AA),\omega)$ contains a maximal subspace which is a direct sum of irreducible representations. This subspace is called the \emph{discrete spectrum}, and an automorphic representation occuring here is called \emph{discrete}. The orthogonal complement of this subspace is the \emph{continuous spectrum}. Langlands identified the continuous spectrum with `Eisenstein series'; it is the direct integral of families of representations induced from parabolic subgroups of $G(\AA)$. The discrete spectrum, in turn, also decomposes as the direct sum of the \emph{cuspidal} and the \emph{residual} spectrum. The cuspidal spectrum is defined as the subspace spanned by functions $f$ such that the integral over $N(\Q)\backslash N(\AA)$ of $f$, and all its translates under $G(\AA)$,  vanishes, for $N$ the unipotent radical of any proper parabolic subgroup. Langlands proved that the residual spectrum is spanned by the residues of Eisenstein series, and that all residual representations  are  quotients of representations induced from a parabolic subgroup.

Any irreducible automorphic representation $\pi$ of $G(\AA)$ is a completed (restricted) tensor product of local representations $\pi_p$ of $G(\Q_p)$, where $p$ ranges over the prime numbers, and an archimedean component $\pi_\infty$. Let $K_p \subset G(\Q_p)$ be a special maximal compact subgroup. We say that $\pi$ is \emph{spherical} at $p$ if $\pi_p$ contains a nonzero vector fixed by $K_p$, in which case this vector will be unique up to a nonzero scalar. The representation $\pi$ is spherical at all but finitely many primes. The word `restricted' in the first sentence of this paragraph means that the component of the representation at $p$ should be equal to the spherical vector for all but finitely many $p$. 

The archimedean component $\pi_\infty$ can be identified with an irreducible $(\g,K_\infty)$-module, where $\g$ is the Lie group of $G(\R)$ and $K_\infty \subset G(\R)$ is a maximal compact subgroup. The center of the universal enveloping algebra of $\g$ acts by a scalar on $\pi_\infty$. The resulting map $Z(\mathcal U \g) \to \C$ is called the \emph{infinitesimal character} of $\pi$.

\subsection{Local factors}\cite{borelLfunctions}
Suppose $\pi$ is spherical at $p$. 
We define the \emph{spherical Hecke algebra} $\H_{G,K_p}$ to be the convolution algebra of $K_p$-bi-invariant $\Q$-valued functions on $G(\Q_p)$. This algebra acts on the one-dimensional space of spherical vectors, and $\pi_p$ is uniquely determined by this action. Hence specifying a spherical representation is equivalent to specifying a homomorphism $\H_{G,K_p} \to \C$. We should therefore understand the ring $\H_{G,K_p}$, and this we can do via the \emph{Satake isomorphism}. For this we need the notion of the \emph{dual group}. If $G$ is defined by a root datum, then the dual group $\widehat{G}$ is obtained by switching roots and co-roots, and characters and $1$-parameter subgroups. The Satake isomorphism states that the Hecke algebra $\H_{G,K_p}$ and the ring of virtual representations $K_0(\RR(\widehat{G}))$ become isomorphic after an extension of scalars: one has
$$ \H_{G,K_p} \otimes \C \cong K_0(\RR(\widehat{G})) \otimes \C. $$
In particular a homomorphism $\H_{G,K_p} \to \C$ is identified with a homomorphism $K_0(\RR(\widehat{G})) \to \C$. But the latter is determined by a semisimple conjugacy class $c_p$ in $\widehat{G}(\C)$. (You evaluate such a class on a representation $V$ via $\Tr(c_p \mid V)$.)

Now suppose instead we have an $\ell$-adic (or $\lambda$-adic) representation $\rho \colon \Gal(\overline \Q/\Q) \to \widehat{G}(\overline \Q_\ell)$. For all but finitely many primes, $\rho$ is going to be unramified, which means in particular that the expression $\rho(\mathrm{Frob}_p)$ is well defined up to conjugacy. If we choose an isomorphism $\C \cong \overline \Q_\ell$, then it makes sense to ask whether $c_p$ and $\rho(\mathrm{Frob}_p)$ are conjugate for almost all $p$. If this holds, then we say that $\rho$ is \emph{attached} to the automorphic representation $\pi$. We remark that this definition would make sense also if we replace $\Gal(\overline \Q/\Q)$ by  the absolute Weil group, or the conjectural global Langlands group, as in both cases $\rho(\mathrm{Frob}_p)$ should be well defined up to conjugacy at unramified primes.

By the Chebotarev density theorem, there is at most one Galois representation  attached to a given automorphic representation. The Strong Multiplicity One theorem shows the converse when $G = \GL(n)$, but in general there will be several automorphic representations with the same attached Galois representation. Conjecturally, two automorphic representations will have the same attached Galois representation if and only if they lie in the same `$L$-packet'. However, the notion of a packet has not been rigorously defined in general.

One of many conjectures within the Langlands program says roughly that there should in fact be a bijection between packets of automorphic representations for $G$ and $\ell$-adic Galois representations into the dual group. As stated this conjecture is however false, and making the conjecture precise is a rather delicate matter. For a formulation in terms of the hypothetical Langlands group, see \cite{arthurlgroup}, and for a more restrictive formulation only in terms of Galois representations, see \cite{buzzardgee}.

Often one fixes once and for all $r \colon \widehat{G} \hookrightarrow \GL(n)$. Then the conjugacy class $c_p$ can be described by specifying an $n \times n$ diagonal matrix $\mathrm{diag}(t_1,\ldots,t_n)$. The numbers $t_i$ are called the \emph{Langlands parameters} of $\pi$ at $p$. Moreover, one can then attach an $L$-function to any automorphic representation. At a prime $p$ where $\pi$ is spherical, the local $L$-factor is given by 
$$ \det(\mathbf 1_n - p^{-s} r(c_p))^{-1}.$$
On the other hand, given $r$ we also obtain from $\rho$ an $n$-dimensional $\ell$-adic Galois representation, which also has an attached $L$-function. Thus the Langlands parameters can be identified with the Frobenius eigenvalues of the attached Galois representations. Usually the notion of $\rho$ being attached to $\pi$ is defined in terms of an equality of $L$-functions, but $L$-functions will play only a minor role in this paper.

\subsection{Shimura varieties} \cite[Kapitel II]{deligne-shimuravarieties,harderbook}\label{shimura} For $G$ as above, suppose that $h \colon \mathrm{Res}_{\C/\R} \mathbb G_m \to G_{/\R}$ is a homomorphism satisfying the axioms 2.1.1.1--2.1.1.3 of \cite{deligne-shimuravarieties}. Let $K_\infty$ be the stabilizer of $h$ in $G(\R)$. Let $K_\f$ be any compact open subgroup of $G(\AA_\f)$. For $K = K_\f \times K_\infty$ we can consider the quotient 
$$ S_{K} = G(\Q) \backslash G(\AA) / K = G(\Q) \backslash X \times G(\AA_\f) / K_\f ,$$
the \emph{Shimura variety} associated to $K$. Here $X = G(\R)/K_\infty$. For $K_\f$ small enough $S_K$ is, in fact, a smooth algebraic variety which is naturally defined over a number field (the reflex field), but in the case we will consider in this paper we will actually need to think of $S_K$ as an orbifold or Deligne--Mumford stack. 

\begin{ex} \label{Ag} Siegel modular varieties are Shimura varieties. Let $G= \GSp(2g)$ and put $$h(x+iy) = \small\begin{bmatrix}
x I_g & y I_g \\ -yI_g & xI_g 
\end{bmatrix}.$$ Then $X = \mathfrak H_g \sqcup \overline{\mathfrak H}_g$ is the union of Siegel's upper half space and its complex conjugate. If we choose $K_\f = G(\widehat \Z)$, then $G(\AA_\f) = G(\Q) \cdot K_\f$ and 
$$ S_K = G(\Q) \backslash X \times G(\AA_\f) / K_\f \cong (G(\Q) \cap K_\f) \backslash X = G(\Z) \backslash X. $$
Now $G(\Z) \backslash X$ is naturally isomorphic to the stack $\A_g$ parametrizing principally polarized abelian varieties of dimension $g$. Had we chosen $K_\f$ smaller, $S_K$  would instead be a disjoint union of finite covers of $\A_g$, parametrizing abelian varieties with `$K_\f$-level structure'.\end{ex}

Let $\V$ be an irreducible finite dimensional rational representation of $G$. To $\V$ we can attach a local system on $S_K$, which we also denote $\V$. As the reader may already have noticed, we (sloppily) use `local system' as a catch-all term to describe several different structures: we obtain a locally constant sheaf of $\Q$-vector spaces on the topological space $S_K(\C)$ which in a natural way underlies a variation of Hodge structure; moreover, $\V \otimes \Q_\ell$ can (for any $\ell$) be identified with the base change of a smooth $\ell$-adic sheaf on ${S_K}$ over the reflex field. The \'etale cohomology groups of said $\ell$-adic sheaves are (after base changing to $\C$) related to the ordinary singular cohomology groups by a comparison isomorphism, and we may think informally of $\V$ as a `motivic sheaf' and $H^\bullet(S_K,\V)$ as a `mixed motive' with a compatible system of $\ell$-adic and Hodge-theoretic realizations.

\subsection{Decomposing cohomology} \cite{arthurL2} In this subsection we will find the need to compare several different cohomology theories. We will use the phrase `ordinary cohomology' to refer to the usual cohomology of the topological space $S_K(\mathbf C)$.

The spectral decomposition of $L^2(G(\Q)\backslash G(\AA))$ contains much information about the cohomology of Shimura varieties for $G$.  The connection to automorphic representations is most transparent if we work transcendentally and consider the sheaf $\V \otimes \C$ on $S_K(\C)$. Then one can consider instead of the usual de Rham complex the complex of forms $\omega$ such that $\omega$ and $d\omega$ are square integrable; the cohomology of this complex is called the \emph{$L^2$-cohomology}.  The $L^2$-cohomology has an interpretation in terms of $(\g,K_\infty)$-cohomology:
$$ H^\bullet_{(2)}(S_K(\C),\V \otimes \C) \cong H^\bullet(\g,K_\infty; \V \otimes L^2(G(\Q)\backslash G(\AA))^{K_\f}. $$ 
According to \cite[Section 4]{borelcasselman}, the contribution from the continuous spectrum to the $(\mathfrak g, K_\infty)$-cohomology vanishes in many natural cases (including all Shimura varieties); in fact, the contribution is non-zero if and only if the $L^2$-cohomology is infinite dimensional. In particular, we may in our case replace $L^2(G(\Q)\backslash G(\AA))$ by the direct sum  $\bigoplus_\pi m(\pi) \pi$ over the discrete spectrum, giving instead the expression
$$ H^\bullet_{(2)}(S_K(\C),\V\otimes \C) \cong \bigoplus_{\pi \text{ disc.}} m(\pi) \pi_\f^{K_\f} \otimes H^\bullet(\g,K_\infty;\V \otimes \pi_\infty). $$ 
In this decomposition, each $\pi_\f^{K_\f}$ is a module over the Hecke algebra, giving the cohomology a Hecke action. Each $H^\bullet(\g,K_\infty;\V \otimes \pi_\infty)$ has a natural $(p,q)$-decomposition, defining a pure Hodge structure on each cohomology group $H^k_{(2)}(S_K(\C),\V)$. 

We say that an automorphic representation $\pi$ is \emph{cohomological} if there exists a representation $\V$ for which $H^\bullet(\g,K_\infty;\V \otimes \pi_\infty) \neq 0$.  Wigner's lemma gives a necessary condition for this nonvanishing of $(\g,K_\infty)$-cohomology, namely that $\pi_\infty$ and $\V^\vee$ (denoting the contragredient) have the same infinitesimal character. For cohomological representations, the infinitesimal character is the book-keeping device that tells you to which local system the automorphic representation will contribute $L^2$-cohomology.

The natural map from $L^2$-cohomology to ordinary cohomology is in general neither injective nor surjective. One can however also define the \emph{cuspidal cohomology} as the direct summand 
$$ H^\bullet_{{\cusp}} (S_K(\C),\V\otimes \C) \cong \bigoplus_{\pi \text{ cusp.}} m(\pi) \pi_\f^{K_\f} \otimes H^\bullet(\g,K_\infty;\V \otimes \pi_\infty) $$ 
of the $L^2$-cohomology, and it  injects naturally into the ordinary cohomology \cite[Corollary 5.5]{borelstablereal}.

Finally one can consider the \emph{inner cohomology}, which is defined as 
$$ H^\bullet_!(S_K,\V) = \mathrm{Image}(H^\bullet_c(S_K,\V) \to H^\bullet(S_K,\V)).$$ When we extend scalars to $\C$, the inner cohomology is sandwiched between the cuspidal and the $L^2$-cohomology. Indeed, the map from compactly supported cohomology to ordinary cohomology always factors through the $L^2$-cohomology, since the orthogonal projection of a closed compactly supported form to the space of harmonic forms is square integrable. This shows that the inner cohomology is a subquotient of the $L^2$-cohomology. On the other hand, the aforementioned result of Borel shows that the cuspidal cohomology injects into the inner cohomology. 

The `complement' of the inner cohomology is called the \emph{Eisenstein cohomology}. Formally, it is defined as the cokernel of $H^\bullet_c(S_K,\V) \to H^\bullet(S_K,\V)$. One could also consider the kernel, which gives the \emph{compactly supported Eisenstein cohomology}. We denote these $H^\bullet_\Eis$ and $H^\bullet_{c,\Eis}$, respectively. We will consider $G = \GSp(2g)$, in which case each local system $\V$ is isomorphic to its dual, up to a twist by the multiplier. Indeed, the restriction of the representation $\V$ to $\Sp(2g)$ satisfies $\V \cong \V^\vee$, with the symplectic pairing providing the isomorphism. 
In this case, we see that either one of $H^\bullet_\Eis$ and $H^\bullet_{c,\Eis}$ determines the other via Poincar\'e duality.

The Zucker conjecture, proven independently in \cite{zuckerconj1,zuckerconj2}, gives an isomorphism between the $L^2$-cohomology of $S_K$ and the intersection cohomology of the Baily--Borel--Satake compactification $\overline S_K$:
$$ H^\bullet_{(2)}(S_K(\C),\V\otimes \C) \cong \IH^\bullet(\overline S_K(\C),j_{!\ast} \V\otimes \C), $$
where $j : S_K \to \overline S_K$ is the inclusion, and $j_{!\ast}$ denotes the intermediate extension. This isomorphism is compatible with the Hecke algebra action.  But the intersection cohomology makes sense algebraically, and we can decompose the intersection cohomology of $\V$ into irreducible Hecke modules already over some number field $F$. We thus get a decomposition
$$
\IH^\bullet(\overline S_K,j_{!\ast}\V\otimes F) = \bigoplus_{\pi_\f} \pi_\f^{K_\f} \otimes H^\bullet(\pi_\f). $$
Here the sum runs over the finite parts of all discrete automorphic representations, and $H^\bullet(\pi_\f) \otimes_F \C$ is isomorphic to $\bigoplus_{\pi_\infty} m(\pi_\f \otimes \pi_\infty) H^\bullet(\g,K_\infty;\V \otimes \pi_\infty)$. For any non-archimedean place $\lambda$ of $F$ we also get a structure of $\lambda$-adic Galois representation on each $H^\bullet(\pi_\f) \otimes F_\lambda$ by a comparison isomorphism with the \'etale intersection cohomology. If we do not insist on a decomposition into absolutely irreducible Hecke modules we can take $F=\Q$, as in Theorem \ref{mainthm}, where we e.g.\ consider a summand corresponding to all cusp forms of given weight, instead of a decomposition into Galois representations attached to individual cusp forms. See \cite[Conjecture 5.2]{blasiusrogawskimotives} for a conjectural formula expressing $H^\bullet(\pi_\f) \otimes F_\lambda$ in terms of Galois representations attached to $\pi$.

\section{The case of $\A_2$}

Consider again the stack $\A_2$ of principally polarized abelian surfaces. As in Example \ref{Ag}, we may think of it as a Shimura variety for $\GSp(4)$. However, we would prefer to work with $G = \PGSp(4)$, and there is a minor issue here. If we put $K_\f = G(\widehat{\Z})$, then the corresponding Shimura variety is  $S_K = \PGSp(4,\Z) \backslash (\mathfrak{H}_2 \coprod \overline{\mathfrak{H}}_2)$, which fails to be isomorphic to $\A_2$ as a \emph{stack}. Indeed every point of $\A_2$ has $\pm 1$ in its isotropy group, but a general point of $S_K$ has trivial isotropy. The projection $\GSp(4) \to \PGSp(4)$ defines a map $\pi \colon \A_2 \to S_K$ which induces an isomorphism on coarse moduli spaces, but which is a $\mu_2$-gerbe in the sense of stacks. 

The finite dimensional irreducible representations of $G$ are indexed by integers $a \geq b \geq 0$ for which $a+b$ is even. The local systems on $S_K$ obtained in this way are strongly related to the local systems $\V_{a,b}$ that we defined in Section \ref{mainresults}. Specifically, if $a+b$ is even, then we may Tate twist the local system $\V_{a,b}$ on $\A_2$ to be a weight zero variation of Hodge structure/$\ell$-adic sheaf; its pushforward under $\pi$ is the one that is naturally attached to an irreducible representation of $\PGSp(4)$. Since $\mathrm R \pi_ \ast \V_{a,b} = \pi_\ast \V_{a,b}$, it will suffice to compute the cohomology of the local systems on $S_K$. From now on we tacitly identify the local systems on $\A_2$ and on $S_K$ with each other. 

In this section we will see how the results in \cite{flicker} allow the computation of the cuspidal and intersection cohomology of these local systems on $\A_2$. Let me emphasize that as mentioned in the above paragraph, by our definition the  $\V_{a,b}$ are Weil sheaves of weight $a+b$; this is the \emph{cohomological normalization}, which is the most natural from the point of view of algebraic geometry. There is also the \emph{unitary normalization}, where $\V_{a,b}$ has weight $0$, which is used in Flicker's work. If $a+b$ is even, as in our case, then the two differ only by a Tate twist.  We will from now on always make this Tate twist whenever we quote results from Flicker's book, without explicitly mentioning it.

Since $\A_2$ is the complement of a normal crossing divisor in a smooth proper stack over $\Spec(\Z)$, and the local systems $\V_{a,b}$ are also defined over $\Spec(\Z)$, the cohomology groups $H^\bullet(\A_2,\V_{a,b} \otimes \Q_\ell)$ must define Galois representations of a very special kind: they are unramified at every prime $p \neq \ell$ and crystalline at $\ell$. The same phenomenon is clear also on the automorphic side. If $\pi_\f^{K_\f} \neq 0$ and $K_\f = G(\widehat{\Z})$, then $\pi_\f$ must be spherical at all primes by definition  since $G(\Z_p)$ is a special maximal compact subgroup of $G(\Q_p)$. Conversely if $\pi_\f$ is spherical everywhere then $\pi_\f^{K_\f}$ is exactly $1$-dimensional.

Considering $\PGSp(4)$ rather than $\GSp(4)$ is the same as only considering automorphic representations of $\GSp(4)$ with trivial central character. The reason we can do this is that we are considering only the completely unramified case (i.e.\ the case of the full modular group); in general, the image of a congruence subgroup of $\GSp(4)$ in $\PGSp(4)$ will no longer be a congruence subgroup. We restrict ourselves to $\PGSp(4)$ in this paper as this is the situation considered in \cite{flicker}.

We note that Flicker's work assumes that all automorphic representations $\pi$ occuring are elliptic at at least three  places. This is explained in Section I.2g of Part 1 of the book. This assumption is present in order to replace Arthur's trace formula with the simple trace formula of \cite{flickerkazhdan}. However, he also notes that this  assumption is only present in order to simplify the exposition --- the same results can be derived assuming only that $\pi$ is elliptic at a single real place, using the same ideas used to derive the simple trace formula in \cite{flickerkazhdan}, as detailed in \cite{laumoncompositio,laumonasterisque}. In particular Flicker's classification of the cohomological part of the discrete spectrum carries through (an archimedean component which is cohomological is elliptic).

We begin by determining $\IH^\bullet(\overline \A_2,j_{!\ast}\V_{a,b})$. This amounts to determining all representations in the discrete spectrum of $\PGSp(4)$ which are spherical at every finite place and cohomological, and for each of them the corresponding Galois representation $H^\bullet(\pi_f)$. All these things are described very precisely by Flicker. Then we shall see that $H^\bullet_{\cusp}(\A_2,\V_{a,b})$ is well defined as a subspace of the \'etale intersection cohomology, and that it coincides with the inner cohomology.

\subsection{The Vogan--Zuckerman classification} Recall that an automorphic representation $\pi_\f \otimes \pi_\infty$ is cohomological if $\pi_\infty$ has nonzero $(\g,K_\infty)$-cohomology with respect to some finite dimensional representation $\V$. If $\pi_\infty$ is in the discrete series, then $\pi$ is always cohomological. The cohomological representations which are not in the discrete series can be determined by \cite{vz}. We recall from \cite[293]{taylor3fold} the result for $\GSp(4)$:

In the regular case there are no cohomological ones apart from the two discrete series representations, which we denote by $\pi^H$ and $\pi^W$ (we omit the infinitesimal character from the notation). The former is in the holomorphic discrete series and the latter has a Whittaker model. Both have $2$-dimensional $(\g,K_\infty)$-cohomology, concentrated in degree 3: their Hodge numbers are $(3,0)$ and $(0,3)$, and $(2,1)$ and $(1,2)$, respectively. 

The representations $\pi$ with $\pi_\infty = \pi^H$ are correspond bijectively to cuspidal Siegel modular eigenforms. 
If $F$ is a holomorphic modular form on Siegel's upper half space of genus $g$, then by the Strong Approximation theorem it defines a function on $G(\AA)$, where $G=\GSp(2g)$. If it is modular for the full modular group, then we obtain a function with trivial central character. The subspace spanned by all right translates of this function is the sought for automorphic representation (or a sum of several copies of it). Conversely, any automorphic representation $\pi$ with archimedean component in the holomorphic discrete series determines uniquely a holomorphic vector-valued cusp form by considering the one-dimensional space of lowest $K_\infty$-type in $\pi_\infty$, and $\pi_\f^{K_\f}$ being one-dimensional forces it to be an eigenvector for all Hecke operators. See \cite{asgarischmidt} for more details.

For singular weights there are further possibilities. If $b=0$ there is a  unitary representation $\pi^1$ whose $(\g,K_\infty)$-cohomology is 2-dimensional in degrees $2$ and $4$, with Hodge types $(2,0)$, $(0,2)$, $(3,1)$ and $(1,3)$. 

If $a=b$ there are two unitary representations $\pi^{2+}$ and $\pi^{2-}$. One is obtained from the other by tensoring with the sign character. Both have $1$-dimensional $(\g,K_\infty)$-cohomology in degrees $2$ and $4$, with Hodge types $(1,1)$ and $(2,2)$. 

Finally if $a=b=0$ we must in addition consider one-dimensional representations, which have cohomology in degrees $0$, $2$, $4$ and $6$: we will ignore this case. 

\subsection{Packets and multiplicities}In Flicker's book, the discrete spectrum of $\PGSp(4)$ is partitioned into `packets' and `quasi-packets', and he conjectures that these coincide with the conjecturally defined $L$-packets and $A$-packets. However, in the totally unramified case the situation simplifies. In general, the (conjectural) $A$-packets are products of \emph{local} $A$-packets, which specify the possible local components $\pi_v$. The local packets at non-archimedean $v$ are expected to have exactly one  spherical member. Since we are only going to consider representations which are spherical at \emph{every} finite place, we thus see that $\pi$ and $\pi'$ will be in the same $A$-packet if and only if they are in the same $L$-packet if and only if $\pi_\f \cong \pi_\f'$.  For this reason we simply write \emph{packet} everywhere in what follows.

In Flicker's classification there are five types of automorphic representations in the discrete spectrum. In the first three types, the corresponding packets are \emph{stable}: each representation in the packet occurs with multiplicity exactly $1$ in the discrete spectrum. Types $4$ and $5$, however, are \emph{unstable}. This means that the multiplicities are not constant over the packets: in general some representations in the packet occur with multiplicity $0$ and others with multiplicity $1$. Flicker \cite[Section 2.II.4]{flicker} gives explicit formulas for the multiplicities of the representations in the packet.

In general there are local packets at each prime $p$ in the unstable case, which consist of either one or two elements. We write such local packets as $\{\Pi_p^+\}$ and $\{\Pi_p^+,\Pi_p^-\}$, respectively. An element of the global packet is specified by choosing an element of the local packet at each  $p$. All but finitely many of the local packets will be singletons, so each packet is finite. When $p=\infty$ we always have $\Pi_p^- = \pi^H$. If $\pi$ lies in an unstable packet, its multiplicity in the discrete spectrum depends only on the parity of the number of places $p$ where $\pi_p = \Pi_p^-$. 

However, as we have already mentioned, the local packet contains only one element for a prime $p$ where $\pi$ is spherical. More generally, certain local representations need to be discrete series in order for $\Pi_p^-$ to be nonzero. Since we are in the level $1$ case, this means that the representations in the packet can differ only in their archimedean component, and the multiplicity formulas simplify significantly: they depend only on whether or not $\pi_\infty = \pi^H$. 

To each discrete spectrum automorphic representation $\pi$ one can attach a $4$-dimensional Galois representation whose Frobenius eigenvalues at $p$ are given by the Langlands parameters at $p$ of $\pi$. (Here we fix the 4-dimensional spin representation of $\mathrm{Spin}(5)$, the dual group of $\PGSp(4)$.) If $\pi$ is in a stable packet, then $H^\bullet(\pi_\f)$ is $4$-dimensional and coincides with this attached representation. In the unstable case, the attached Galois representation is always a sum of two $2$-dimensional pieces, and $H^\bullet(\pi_\f)$ is given by one of these two summands. Which of the two halves contributes nontrivially is decided by a formula similar to the multiplicity formula, see \cite[Part 2, V.2]{flicker}. In particular it again has the feature that it depends on the parity of the number of places where $\pi_p = \Pi_p^-$, and simplifies significantly in the completely unramified case.

\subsection{The discrete spectrum of $\PGSp(4)$}

The discrete spectrum automorphic representations which can contribute nontrivially to $H^\bullet_{(2)}(\A_2,\V_{a,b})$ have an archimedean component with infinitesimal character $(a,b)+(2,1)$. A complete classification into five types is given in \cite[Theorem 2, pp. 213--216]{flicker}. We deal with each type separately. This classification is the same as the one announced by Arthur for $\GSp(4)$ \cite{arthurclassification}, except that the ones of Howe--Piatetski-Shapiro-type do not appear.  

In what follows we write in parentheses the names assigned to these families by Arthur. 

\subsubsection{Type 1 (General type)}

These are exactly the ones that lift to cuspidal representations of $\PGL(4)$. 

Each of these lies in a packet of cardinality $2$, where the elements in the packet are distinguished by their archimedean component: one is in the holomorphic discrete series and the other has a Whittaker model. Both elements of the packet occur with multiplicity $1$ in the discrete spectrum. Packets of this type correspond bijectively to vector valued cuspidal Siegel eigenforms which are neither endoscopic (a Yoshida-type lifting) nor CAP (a Saito--Kurokawa-type lifting). The contribution from this part of the discrete spectrum to $\IH^\bullet(\overline \A_2,j_{!\ast}\V_{a,b})$ is concentrated in degree $3$ and is the sum of the Galois representations attached to the Siegel cusp forms. We shall see that the Yoshida-type liftings do not occur in level 1. We denote this contribution to the cohomology by $\S_{a-b,b+3}^\gen$.

\subsubsection{Type 2 (Soudry type)}

These packets are singletons, and the archimedean component is $\pi^1$, and will therefore not occur unless $b=0$. Every packet is obtained by a lifting from a cuspidal representation $\Pi$ of $\GL(2)$, corresponding to a cusp eigenform of weight $a+1$ whose central character $\xi$ is quadratic, $\xi \neq 1$, and $\xi \Pi = \Pi$. This is obviously impossible in level $1$ for several reasons: for one, $a$ must be even, and there are no modular forms of odd weight for $\SL(2,\Z)$.  

\subsubsection{Type 3 (One-dimensional type)}

These are the representations with $\pi_\infty$ $1$-dimensional and will only occur when $a=b=0$; for our purposes this case can clearly be ignored.

\subsubsection{Type 4 (Yoshida type)}

This is the first unstable case. All these $\pi$ have $\pi_\infty \in \{\pi^H,\pi^W\}$ and their $L$-function is the product of $L$-functions attached to cusp forms for $\GL(2)$. For each pair of cuspidal automorphic representations $\Pi_1$ and $\Pi_2$ of $\PGL(2)$ whose weights are $a+b+4$ and $a-b+2$, respectively, there is a packet $\{\pi\}$ of Yoshida type. %As explained in our discussion of multiplicity formulas this packet will in our case have cardinality two, consisting of
As explained earlier, the fact that we are in the unramified case implies that members of the packet can only differ in their archimedean component, so we should consider only
 $\pi_{\f} \otimes \pi^H$ and $\pi_\f \otimes \pi^W$. The multiplicity formula simplifies (since we are in the unramified case) to 
$$ m(\pi_\f \otimes \pi^H) = 0, \qquad m(\pi_\f \otimes \pi^W)=1,$$
and so $\pi_\f \otimes \pi^W$ will contribute a $2$-dimensional piece of the cohomology in degree 3. The trace of Frobenius on this part of cohomology is also calculated by Flicker and we find the Galois representation $\rho_{\Pi_2} \otimes \Q_\ell(-b-1)$, where $\rho_{\Pi_2}$ is the $2$-dimensional representation attached to $\Pi_2$. Summing over all $\Pi_1 $ and $\Pi_2$ this part therefore contributes
$$ s_{a+b+4}\S_{a-b+2} (-b-1) $$
to $\IH^3(\overline \A_2,j_{!\ast}\V_{a,b})$.

We note in particular that there are no Yoshida-type liftings to Siegel cusp forms in level 1: these would correspond to a $\pi$ with $\pi_\infty = \pi^H$ and multiplicity $1$. 

The required liftings and multiplicity formulas for the endoscopic case have also been established for $\GSp(4)$ in \cite[Theorem 5.2]{weissauerbook}.

\subsubsection{Type 5 (Saito--Kurokawa type)}

This case appears only when $a=b$. Here there are four possible archimedean components: $\pi^H, \pi^W, \pi^{2+}$ and $\pi^{2-}$. Every packet contains precisely one of $\pi^{2+}$ and $\pi^{2-}$. For each cuspidal automorphic representation $\Pi$ of $\PGL(2)$ of weight $a+b+4$ and for $\xi \in \{1,\sgn\}$ we get a Saito--Kurokawa packet $\{\pi\}$. Since we are in level $1$, we can ignore the character $\xi$ (it must be trivial), which means that $\pi^{2-}$ will not appear. 

I should also say that there is a minor error at this place of Flicker's book. Flicker states that the Langlands parameters at a place $u$ are (his notation) 
$$ \mathrm{diag}(\xi_u q_u^{1/2}z_{1u}, \xi_u q_u^{1/2}z_{2u}, \xi_u q_u^{-1/2}z_{2u}, \xi_u q_u^{-1/2}z_{1u}) $$
when they should be
$$ \mathrm{diag}(z_{1u}, \xi_u q_u^{1/2}, \xi_u q_u^{-1/2}, z_{2u}). $$

Let us then consider the multiplicities, which again simplify since we are in the level 1 case: we find 
$$ m(\pi_\f \otimes \pi_\infty) = \frac 1 2 \left( 1 + \varepsilon(\Pi,{\tfrac 1 2} )\cdot (-1)^{n}\right) $$
where $n = 1$ if $\pi_\infty = \pi^H$ and $n=0$ otherwise, and $\varepsilon(\Pi, \frac{1}{2}) = (-1)^k$ if $\Pi$ is attached to a cusp form of weight $2k$.

We thus see that if $a=b$ is odd, the only representation in the packet with nonzero multiplicity is $\pi_\f \otimes \pi^H$, which should correspond to a Siegel modular form. The Siegel modular forms obtained in this way are precisely the classical Saito--Kurokawa liftings, and the contribution in this case is exactly $\S_{a+b+4}$.

For $a = b$ even we could a priori have both $\pi_\f \otimes\pi^W$ and $\pi_\f \otimes \pi^{2+}$ with nonzero multiplicity. But we can see by studying the Frobenius eigenvalues that $\pi^W$ will not appear. Indeed, the representation $\pi_\f \otimes \pi^W$ would contribute to the intersection cohomology  in degree $3$, as we see from the $(\g,K_\infty)$-cohomology of $\pi^W$. Then its Frobenius eigenvalues are pure of weight $a+b+3$. But the Frobenius eigenvalues at $p$ will be $p^{b+1}$ and $p^{b+2}$, as determined by Flicker, a contradiction. On the other hand, we know that $\pi_\f \otimes \pi^{2+}$ is automorphic:  it is the Langlands quotient of 
 $$ \mathrm{Ind}_{P(\AA)}^{\PGSp(4,\AA)} \left( \Pi \otimes 1 \right),$$
 where $P$ is the Siegel parabolic (whose Levi component is $\PGL(2) \times \GL(1)$),
and the multiplicity formula shows that it has multiplicity $1$ in the discrete spectrum. The representations of this form will contribute a term $s_{a+b+4} \Q_\ell(-{b-1}) $ to $\IH^2(\overline \A_2,j_{!\ast}\V_{a,b})$ and $s_{a+b+4}\Q_\ell(-{b-2}) $ to $\IH^4(\overline \A_2,j_{!\ast}\V_{a,b})$.

\begin{rem}That $\pi_\infty=\pi^W$ does not occur in the Saito--Kurokawa case is mentioned as a conjecture of Blasius and Rogawski in \cite[Section 6]{tilouinesiegelthreefold}. The argument above will prove this conjecture for $\PGSp(4)$. Probably a proof for $\GSp(4)$ in general can be obtained by a similar argument, or by considering the possible Hodge numbers of $H^\bullet(\pi_\f)$.  \end{rem}

\subsection{The inner cohomology and the proof of the main theorem}
From what we have seen so far, we can completely write down the $L^2$-cohomology and the intersection cohomology of any local system on $\A_2$. Summing up the contributions from all parts of the discrete spectrum, we see that 
$$\IH^3(\overline \A_2,j_{!\ast}\V_{a,b})\cong \S^{\gen}_{a-b,b+3} +s_{a+b+4} \S_{a-b+2} + 
\begin{cases} \S_{a+b+4} & a=b \text{ odd,} \\ 0 & \text{otherwise.}\end{cases} $$
The cohomology vanishes outside the middle degree in all cases except when $a=b$ is even, when we have 
$$ \IH^2(\overline \A_2,j_{!\ast}\V_{a,b}) \cong \begin{cases} s_{a+b+4} \Q_\ell(-b-1) & a=b \text{ even,} \\ 0 & \text{otherwise,}\end{cases}$$
and $\IH^4(\overline \A_2,j_{!\ast}\V_{a,b}) \cong \IH^2(\overline \A_2,j_{!\ast}\V_{a,b})(-1).$

Note that the sum
$$\S^{\gen}_{a-b,b+3} +
\begin{cases} \S_{a+b+4} & a=b \text{ odd,} \\ 0 & \text{otherwise,}\end{cases} $$
is exactly what was denoted $\overline \S_{a-b,b+3}$ in Theorem \ref{mainthm}, since there are no Yoshida-type liftings in our case.

If we wish to determine in addition the cuspidal cohomology, then we need to understand which of the above representations are in the residual spectrum. The residual spectrum of $\GSp(4)$ is completely described in \cite[Section 7]{kimresidual}. We see that there is exactly one case above where the representation is residual: namely, the Langlands quotient of $ \mathrm{Ind}_{P(\AA)}^{\PGSp(4,\AA)} \left( \Pi \otimes 1 \right)$ is residual if and only if $L(\Pi, \frac 1 2 ) $ is nonzero. We deduce that the cuspidal cohomology coincides with the $L^2$-cohomology except in degrees $2$ and $4$ when $a=b$ is even, where we have 
$$ H^2_{\cusp}(\A_2,\V_{a,b}) \cong s'_{a+b+4} \Q(-b-1) $$
(so conjecturally, it vanishes) and similarly for $H^4_\cusp$. We also observe that for all packets, either all discrete representations are cuspidal or all are residual, so that the cuspidal cohomology makes sense also as a summand of the \'etale intersection cohomology (a priori it is only a summand in the $L^2$-cohomology), and we can talk about the Galois representation on the cuspidal cohomology.

 The Eisenstein cohomology of any local system on $\A_2$ has been completely determined in any degree, considered as an $\ell$-adic Galois representation up to semi-simplification, in \cite{harder}. From loc.\ cit.\ and the above discussion we may deduce the following.
 
\begin{prop}The natural map $H^\bullet_\cusp(\A_2,\V_{a,b}) \to H^\bullet_!(\A_2,\V_{a,b})$ is an isomorphism for any $a,b$. \end{prop}

\begin{proof}
Recall that  one has 
$$ \IH^k(\overline \A_2,j_{!\ast}\V_{a,b}) = H^k_\cusp(\A_2,\V_{a,b}) \oplus H^k_{\mathrm{res}}(\A_2,\V_{a,b}) $$
and 
$$ W_{k+a+b}H^k(\A_2,\V_{a,b}) = H^k_!(\A_2,\V_{a,b}) \oplus W_{k+a+b} H^k_\Eis(\A_2,\V_{a,b}).$$
Moreover, the map $\IH^k(\overline \A_2,j_{!\ast}\V_{a,b}) \to W_{k+a+b}H^k(\A_2,\V_{a,b})$ is surjective and maps the cuspidal cohomology into the inner cohomology. Hence if $H^k_{\mathrm{res}}(\A_2,\V_{a,b})$ and $W_{k+a+b} H^k_\Eis(\A_2,\V_{a,b})$ have the same dimension, then $H^k_\cusp(\A_2,\V_{a,b}) \to H^k_!(\A_2,\V_{a,b})$ is an isomorphism. 

Now we have seen that $H^k_{\mathrm{res}}(\A_2,\V_{a,b})$ is nonzero only for $k \in \{2,4\}$ and $a=b$ even, so these are the only cases where it is not automatic that $H^\bullet_\cusp(\A_2,\V_{a,b}) \to H^\bullet_!(\A_2,\V_{a,b})$ is an isomorphism. The dimension of $H^k_{\mathrm{res}}(\A_2,\V_{a,b})$ is $s_{a+b+4}-s_{a+b+4}'$ in these cases. From Harder's paper we see that $W_{k+a+b}H^k_\Eis(\A_2,\V_{a,b}) \neq 0$ only for $k =2$ and $a=b$ even, in which case its dimension, too, is $s_{a+b+4}-s_{a+b+4}'$. But then $H^2_\cusp(\A_2,\V_{a,b}) \to H^2_!(\A_2,\V_{a,b})$ is an isomorphism by the preceding paragraph, and then it is an isomorphism also in degree $4$ since both the cuspidal and the inner cohomology satisfy Poincar\'e duality.
\end{proof}

\begin{rem} The equality of dimensions above is not surprising, since Harder explicitly constructs these pure Eisenstein cohomology classes as residues of Eisenstein series associated to cusp forms for $\SL(2,\Z)$ with nonvanishing central value. So in a sense the dimension argument in the preceding theorem is unnecessarily convoluted. See also \cite{residual} which describes in general all possible contributions from the residual spectrum to the Eisenstein cohomology of a Siegel threefold.\end{rem}

 The main theorem of the paper follows from this result, as we now explain.

 \begin{proof}[Proof of Theorem \ref{mainthm}] 
 
 Up to semi-simplification we have $$H^\bullet_c(\A_2,\V_{a,b}) = H^\bullet_! (\A_2,\V_{a,b})\oplus H^\bullet_{c,\Eis}(\A_2,\V_{a,b}),$$ and $H^\bullet_{c,\Eis}(\A_2,\V_{a,b})$ was as already remarked determined in \cite{harder}. By the preceding proposition we have $H^\bullet_!(\A_2,\V_{a,b}) = H^\bullet_\cusp(\A_2,\V_{a,b})$, and the latter has been determined already in this section. Summing up the Eisenstein cohomology and the cuspidal contribution gives the result. \end{proof}

\printbibliography

\end{document}